\documentclass[12pt]{amsart}
\usepackage{amsmath,amssymb,amsthm,mathtools}
\usepackage[all]{xy}
\usepackage{graphics,xcolor}
\usepackage{hyperref}

\newtheorem{Lemma}{Lemma}[section]

\theoremstyle{definition}

\newtheorem{Theorem}[Lemma]{Theorem}
\newtheorem{Definition}[Lemma]{Definition}
\newtheorem{Fact}[Lemma]{Fact}

\newtheorem{Remark}[Lemma]{Remark}
\newtheorem{Corollary}[Lemma]{Corollary}
\newtheorem{Convention}[Lemma]{Convention}

\newcommand{\C}{\mathbb{C}}
\newcommand{\R}{\mathbb{R}}
\newcommand{\lie}[1]{\mathfrak{#1}}
\newcommand{\skipover}[1]{{}}
\newcommand{\diag}{\qopname\relax o{diag}}
\newcommand{\Image}{\qopname\relax o{Im}}

\makeatletter
\@addtoreset{equation}{section} 
\makeatother

\title[Real double flag variety \& Galois cohomology]{Real double flag variety for the symmetric pair $(U(p,p),GL_{p}(\mathbb{C}))$ and Galois cohomology}

\author[K.\,Nishiyama \& T.\,Tauchi]{Kyo Nishiyama$^{1}$ and Taito Tauchi$^{2}$}

\address{$^{1,2}$\; Department of Mathematics, Aoyama Gakuin University, Sagamihara, Japan}

\thanks{K.~N.~is supported by JSPS KAKENHI Grant Number \#{21K03184}.}
\email{kyo@math.aoyama.ac.jp$^1$,
tauchi@math.aoyama.ac.jp$^2$}


\newcommand{\version}{Ver.~0.0}
\newcommand{\setversion}[1]{\renewcommand{\version}{Ver.~{#1}}}
\setversion{1.0 [2024/09/15]}
\setversion{2.0 [2025/04/03 11:59:43 JST]}

\begin{document}

\begin{abstract}
Let
$G$ be the indefinite unitary group $U(p,p)$,
$H\simeq GL_{p}(\mathbb{C})$ its symmetric subgroup,
$P_{S}$ the Siegel parabolic subgroup of $G$, 
and $B_{H}$ a Borel subgroup of $H$.
In this article, 
we give a classification of the orbit decomposition 
$H\backslash (H/B_{H}\times G/P_{S})$ of the real double flag variety 
by using the Galois cohomology 
in the case where $p=2$.
\end{abstract}

\maketitle

\section{Introduction and Main theorem}

Let $G_{\mathbb{C}}$ be a complex reductive algebraic group,
$B_{\mathbb{C}}$ a Borel subgroup of $G_{\mathbb{C}}$.
The orbit decomposition of the flag variety $G_{\mathbb{C}}/B_{\mathbb{C}}$ with respect to
a subgroup $ K_{\C} $ is related to the theory of $ (\lie{g}, K_{\C}) $-modules by the Beilinson-Bernstein localization \cite{BB-loc}, 
when $ K_{\C} $ has finitely many orbits in $ G_{\mathbb{C}}/B_{\mathbb{C}}$.  
Thus, such orbit decomposition in the flag variety is one of central subjects in the representation theory of real reductive groups.

The double flag variety was defined in \cite{KO11-DFV}
as a generalization of the flag variety, and it is ultimately related to the representation theory of $G_{\mathbb{C}}$.

\begin{Definition}
Let $G_{\mathbb{C}}$ be a complex reductive algebraic group, 
$H_{\mathbb{C}}$ a symmetric subgroup of $G_{\mathbb{C}}$, and 
$P_{G_{\mathbb{C}}}$ and $P_{H_{\mathbb{C}}}$ parabolic subgroups of 
$G_{\mathbb{C}}$ and $H_{\mathbb{C}}$, respectively. 
Then, the product 
$\mathfrak{X} \coloneqq H_{\mathbb{C}}/ P_{H_{\mathbb{C}}} \times  G_{\mathbb{C}}/ P_{G_{\mathbb{C}}} $ 
of two flag varieties with  a diagonal $H_{\mathbb{C}}$-action is called \textit{a double flag variety}. 
\end{Definition}

Clearly the above definition makes sense for any ground field instead of $ \C $, and 
in particular we can consider a double flag variety over $ \R $.
We call it a \textit{real double flag variety}.  
There are many interesting results on double flag varieties 
over the complex number field $\mathbb{C}$
from both of the geometric and representation theoretic point of view, e.g., see 
\cite{FN-21-embedding,FN-22-Steinberg,FN-23-AIII}.
It seems that
there are few studies 
of real double flag varieties.  
However, those double flag varieties over $ \R $ seem to have interesting properties and 
we expect useful applications to the geometry, representation theory, combinatorics, and so on.  
For example, for each open orbit in $ \mathfrak{X} $, we can construct an integral kernel operator which intertwines a degenerate principal series of $ G $ and that of $ H $.  Precise results and statements will be appear elsewhere.

In this article,
we give a classification of the orbit decomposition
of a real double flag variety
in some specific case, as a first step toward the general theory of real double flag varieties.
Our main theorem in this article is

\begin{Theorem}
\label{Them-Main-in-intro}
Let
$G$ be the indefinite unitary group 
$U(p,p)$,
$H \simeq GL_{p}(\mathbb{C})$ its symmetric subgroup,
$P$ the Siegel parabolic subgroup
of $G$,
and
$B$ a Borel subgroup of $H\simeq GL_{p}(\mathbb{C})$.
Then,
in the case where $p=2$,
complete representatives of 
the orbit decomposition 
$H\backslash(H/B \times G/P)
\simeq B\backslash G/P$
are given in 
the right-most terms in
Table 
\ref{Fig-list},
where we use the identification via an isomorphism in Lemma \ref{Lem-G/P=LM/GL}.
In particular, we have
$	\# (B\backslash G/P) = 17$. 
\end{Theorem}

Thanks to the theory of the 
Galois cohomology \cite{BS-GC,Serre-GC},
in principle, 
the classification of the orbit decomposition 
$H\backslash(H/B \times G/P)$
of a real double double flag variety can be reduced to that of the complexification 
$H_{\mathbb{C}} \backslash (H_{\mathbb{C}}/B_{\mathbb{C}} \times G_{\mathbb{C}}/P_{\mathbb{C}})$, 
where the subscript $\mathbb{C}$ 
denotes the complexification of those groups.
Although
it seems that
uniform calculations of the Galois 
cohomologies
are difficult in the general setting
of double flag varieties,
it works completely in our setting.
So we prove Theorem \ref{Them-Main-in-intro}
by using the 
Galois cohomology for $ p = 2$.

We have also obtained the same result as Theorem \ref{Them-Main-in-intro}
in general (namely, in the case of $p\geq 3$) by completely different method, 
namely, in terms of the Matsuki duality and the classification of $ K_{\C} $-orbits on $ G_{\C}/B_{\C} $.  
We will give it with details somewhere else.

\skipover{
The outline of this article as follows:
In Section \ref{Section-Preliminary}, we fix some notations about $U(p,p)$ and prove some elementary formulas.
In Section \ref{Section-Orbit-decomp-over-C}, we quote the results about the orbit decomposition 
$H_{\mathbb{C}}\backslash (H_{\mathbb{C}}/B_{\mathbb{C}} \times G_{\mathbb{C}}/P_{\mathbb{C}})$
of the complexification from \cite{FN-21-embedding,FN-23-AIII} and reduce the proof of Theorem \ref{Them-Main-in-intro} to the calculation of the Galois cohomologies.
Finally, we give the proof of Theorem \ref{Them-Main-in-intro} in Section \ref{Section-proof-of-main-thm}.
We also summarize the facts about the Galois cohomology in Appendix \ref{Section-Galois-cohomolgy} for the convenience.
}

\section{Preliminary}\label{Section-Preliminary}


\subsection{Indefinite unitary group}\label{Section-Over-Real}
Let $p\in\mathbb{Z}_{\geq 1}$ and define two matrices $\Sigma_{p}$ and $J_{p}$ by
\begin{equation}
	\Sigma_{p}
	\coloneqq
	\begin{pmatrix}
		0 & -I_{p}
		\\ 
		I_{p} & 0
	\end{pmatrix}
	,\quad
	J_{p}
	\coloneqq 
	\sqrt{-1}
	\Sigma_{p},
	\label{eq-def-Sigma-p-J-p}
\end{equation}
where $I_{p}$
is the identity matrix of size $ p $.
Note that 
$
	J_{p}^{-1}
	=
	J_{p},
	\;
	J_{p}^{*}
	=
	J_{p},
$
where $(\cdot)^{*}$ denotes 
the conjugate transpose.
We define 
the indefinite unitary group
by
\begin{equation}
	G\coloneqq
	U(p,p)
	\coloneqq
	\{g\in GL_{2p}(\mathbb{C})
	\mid 
	g^{*}J_{p}g = J_{p}
	\}.
\label{eq-def-U(p,p)}
\end{equation}
Moreover,
we define a symmetric subgroup $H$ of $G$
and
the Siegel parabolic subgroup $P$ of 
$G$
by
\begin{equation*}
	H
	\coloneqq
	\left\{
\diag(g, (g^{*})^{-1}) 
	\middle|
	g\in GL_{p}(\mathbb{C})
	\right\},
	\quad
	P
	\coloneqq
	\left\{
	\begin{pmatrix}
		g & u
		\\
		0 & (g^{*})^{-1}
	\end{pmatrix}
	\middle|
	\begin{array}{c}
		g \in GL_{p}(\mathbb{C})
		\\
		g^{-1} u \in \operatorname{Her}_p(\mathbb{C})
	\end{array}
	\right\},
\end{equation*}
where we denote $ \diag(a, b) = \begin{pmatrix}
		a & 0
		\\
		0 & b
	\end{pmatrix}
$ (a block diagonal matrix) and 
$\operatorname{Her}_p(\mathbb{C})$
is the space of Hermitian matrices of degree $p$.
We also define a Borel subgroup $B$ of $H$
by
$
	B
	\coloneqq
	\left\{
\diag(b, (b^{*})^{-1})
\skipover{	\begin{pmatrix}
		b & 0
		\\
		0 & (b^{*})^{-1}
	\end{pmatrix}
}
	\middle|
		b\in B_{p}
	\right\} $,
where 
$B_{p}$
is
the subgroup of $GL_{p}(\mathbb{C})$
of upper triangular matrices.

\begin{Definition}
Let $\operatorname{M}_{2p,p}(\mathbb{C})$
be the space of $(2p\times p)$-matrices.
	We define a subset 
	$\operatorname{LM}_{2p,p}^{\circ}$
	of certain full rank matrices in $\operatorname{M}_{2p,p}(\mathbb{C})$
	by
	\begin{equation}
	\operatorname{LM}_{2p,p}^{\circ}
	\coloneqq
		\left\{
		M(C,D)\coloneqq
		\begin{pmatrix}
			C
			\\
			D
		\end{pmatrix}
		\in \operatorname{M}_{2p,p}(\mathbb{C})
		\middle|
		\begin{array}{c}
			C,D\in \operatorname{M}_{p,p}(\mathbb{C})
			\\
			C^{*}D=D^{*}C
			\\
			\operatorname{rank}M(C,D)=p
		\end{array}
		\right\} .
	\label{eq-def-LM-and-M(C,D)}
	\end{equation}
\end{Definition}

\begin{Lemma}\label{Lem-G/P=LM/GL}
We get a $G$-equivariant isomorphism:
$
		G/P
		\simeq 
		\operatorname{LM}_{2p,p}^{\circ}
		/GL_{p}(\mathbb{C}) $.
\end{Lemma}

\begin{proof}
	The flag variety $G/P$ can be identified with 
	the set of all maximally isotropic subspaces of $\mathbb{C}^{2p}$
	because the Siegel parabolic subgroup $P \subset G$ is the stabilizer  
	of a maximally isotropic subspace of $\mathbb{C}^{2p}$ with respect to the indefinite Hermitian inner product defined by $J_{p}$.
	Under this identification, a maximal isotropic subspace corresponds to a matrix in $ \operatorname{LM}_{2p,p}^{\circ} $  whose columns consist of its basis modulo the action of $ GL_p(\C) $ (change of basis matrix).
\end{proof}

\subsection{The complexification}
\label{Section-Over-Complex}
We put 
$
	G_{\mathbb{C}}\coloneqq
	GL_{2p}(\mathbb{C})
$, the complexification of $ G $, 
and define an involution $\tau$ on $G_{\mathbb{C}}$ by
\begin{equation}
	\tau(g)
	\coloneqq
	J_{p}^{-1}
	(g^{*})^{-1}
	J_{p},
\label{eq-def-tau}
\end{equation}
where $J_{p}$ is defined in 
\eqref{eq-def-Sigma-p-J-p}.
Then, we get 
$
	G
	=
	G_{\mathbb{C}}^{\tau}
$ by definition of $ G $.
Moreover,
$(G_{\mathbb{C}},\tau)$
is an algebraic group defined over $\mathbb{R}$ with $ \tau $ as a complex conjugation 
(see Convention \ref{Conv-defined-over-R}).
We also define a symmetric subgroup $H_{\mathbb{C}}\simeq GL_{p}(\mathbb{C})
\times
GL_{p}(\mathbb{C})
$
of $G_{\mathbb{C}}$ by
\begin{equation}
	H_{\mathbb{C}}
	\coloneqq
	\left\{
	\diag(g_{1},g_{2})
\skipover{\coloneqq
	\begin{pmatrix}
		g_{1} & 0
		\\
		0 & g_{2}
	\end{pmatrix}
}
	\middle|
	g_{1},g_{2}\in GL_{p}(\mathbb{C})
	\right\}.
\label{eq-def-H-C}
\end{equation}
Define a $\tau$-stable maximal parabolic subgroup 
$P_{\mathbb{C}}$ of $G_{\mathbb{C}}$,
and a $\tau$-stable Borel subgroup 
$B_{\mathbb{C}}$ of $H_{\mathbb{C}}$ by
\begin{equation}
	P_{\mathbb{C}}
	\coloneqq
	\left\{
	\begin{pmatrix}
		g_{1} & u
		\\
		0 & g_{2}
	\end{pmatrix}
	\middle|
	\begin{array}{c}
		g_{1},g_{2} \in GL_{p}(\mathbb{C})
		\\
		u \in \operatorname{M}_{p,p}(\mathbb{C})
	\end{array}
	\right\},
	\qquad
		B_{\mathbb{C}}
	\coloneqq
	\left\{
	\begin{pmatrix}
		b & 0
		\\
		0 & c
	\end{pmatrix}
	\middle|
	\begin{array}{cc}
		b\in B_{p}
		\\
		c\in \overline{B}_{p}
	\end{array}
	\right\},
\label{eq-def-P_C-B_C}
\end{equation}
where $\overline{B}_{p}$ is the subgroup of $GL_{p}(\mathbb{C})$ of lower triangular matrices.
Then, we have
\begin{equation}
	G_{\mathbb{C}}^{\tau}=G,
	\quad
	H_{\mathbb{C}}^{\tau}=H,
	\quad
	P_{\mathbb{C}}^{\tau}=P,
	\quad
	B_{\mathbb{C}}^{\tau}=B,
	\label{eq-def-G-H-P-B}
\end{equation}
where the groups on the right-hand sides are real forms 
defined in Section \ref{Section-Over-Real}.

\begin{Definition}
	We define a subset
	$\operatorname{M}_{2p,p}^{\circ}$
	of $\operatorname{M}_{2p,p}(\mathbb{C})$
	by
	\begin{equation}
	\operatorname{M}_{2p,p}^{\circ}
	\coloneqq
		\left\{
		M(C,D)
		\in \operatorname{M}_{2p,p}(\mathbb{C})
		\middle|
			C,D\in \operatorname{M}_{p,p}(\mathbb{C}),\:
			\operatorname{rank}M(C,D)=p
		\right\},
	\label{eq-def-M-circ}
	\end{equation}
	where we use the notation $ M(C,D) $ given in \eqref{eq-def-LM-and-M(C,D)}.
\end{Definition}

\begin{Lemma}\label{eq-Lem-G_C/P_C=M/GL}\label{Lem-G_C/P_C=M/GL}
With the above notation,
	we have a $G_{\mathbb{C}}$-equivariant isomorphism:
$
		G_{\mathbb{C}}/P_{\mathbb{C}}
		\simeq 
		\operatorname{M}_{2p,p}^{\circ}
		/GL_{p}(\mathbb{C}) $.
Note that $ G_{\mathbb{C}}/P_{\mathbb{C}} $ is the Grassmannian of $ p $-dimensional subspaces in $ \C^{2p} $.  
\end{Lemma}

\begin{proof}
The same proof as
Lemma \ref{Lem-G/P=LM/GL}
works.
\end{proof}

\begin{Lemma}
\label{Lem-formulae-of-tau-and-H}
	For $\diag(g_{1},g_{2})\in H_{\mathbb{C}}$,
	we have
	\begin{equation}
	\diag(g_{1},g_{2})^{-1}
	=
	\diag(g_{1}^{-1},g_{2}^{-1}),
	\quad
		\tau(\diag(g_{1},g_{2}))
		=
		\diag((g_{2}^{*})^{-1},(g_{1}^{*})^{-1}).
		\label{eq-Lem-H(b,c)^{-1}}
	\end{equation}
Thus, 
in particular,
we have
(see 
\eqref{eq-def-G^{-gamma}} for the notation)
\begin{equation*}
	\tau(\diag(g_{1},g_{2}))^{-1}
	\diag(g_{1},g_{2})
	=
	\diag(g_{2}^{*}g_{1},g_{1}^{*}g_{2}),
	\quad
	H_{\mathbb{C}}^{-\tau}
	=
	\{
	\diag(g,g^{*})
	\mid
	g\in GL_{p}(\mathbb{C})
	\}.
\end{equation*}
\end{Lemma}

\begin{proof}
The proof is straightforward by calculations.
\skipover{
The first equation of 	
\eqref{eq-Lem-H(b,c)^{-1}}
is a consequence of the definition
\eqref{eq-def-H-C}
of $\diag(g_{1},g_{2})$.
By the definition
we get 
\begin{align*}
	\tau(\diag(g_{1},g_{2}))
=
	J_{p}^{-1}
	(\diag(g_{1},g_{2})^{*})^{-1}
	J_{p}
\skipover{
= 
	\sqrt{-1}
	\begin{pmatrix}
		0 & -1 
		\\ 
		1 & 0
	\end{pmatrix}
	\begin{pmatrix}
		(g_{1}^{*})^{-1} & 0
		\\
		0 & (g_{2}^{*})^{-1}
	\end{pmatrix}
	\sqrt{-1}
	\begin{pmatrix}
		0 & -1 
		\\ 
		1 & 0
	\end{pmatrix}
}
= 
	\begin{pmatrix}
		(g_{2}^{*})^{-1} &  0
		\\ 
		0 & (g_{1}^{*})^{-1}
	\end{pmatrix},
\end{align*}
which is the second equation of 
\eqref{eq-Lem-H(b,c)^{-1}}.
The second sentence of the lemma follows from the formulas of \eqref{eq-Lem-H(b,c)^{-1}}.
}
\end{proof}

\begin{Lemma}
\label{Lem-H^1(B_C,R)=1}
Let $B_{\mathbb{C}}$ and 
$P_{\mathbb{C}}$ be the groups defined in \eqref{eq-def-P_C-B_C}. 
The first Galois cohomologies of $ B_{\C} $ and $ P_{\C} $ are trivial:
\begin{equation}
	H^{1}(\mathbb{R},B_{\mathbb{C}})
	=
	1,
	\quad
	H^{1}(\mathbb{R},P_{\mathbb{C}}) 
        = 1.
	\label{eq-Lem-H^1(B_C,R)-H^1(P_C,R)}
\end{equation}
\end{Lemma}

\begin{proof}
First,
let us prove $ H^{1}(\mathbb{R},B_{\mathbb{C}}) = 1 $.
Define two subgroups of $H_{\mathbb{C}}$
\begin{equation}
    \begin{split}
	T_{\mathbb{C}}
	&\coloneqq
	\left\{
	T(t,s)
	\coloneqq
	\diag\left(
	t
	,
	s
	\right)
	\middle|
	t,s\in 
	GL_{1}(\mathbb{C})^{p}
	\right\}
	,
	\\
	N_{\mathbb{C}}
	&\coloneqq
	\left\{
	N(n_{+},n_{-})
	\coloneqq
	\diag\left(
	n_{+}
	,
	n_{-}	
	\right)
	\middle|
	n_{+} \in N_{\mathbb{C}}^{+},
    n_{-} \in N_{\mathbb{C}}^{-}
	\right\},
    \end{split}
\label{eq-def-T-and-N}
\end{equation}
where we identify $GL_{1}(\C)^{p}$ with the subgroup of $GL_{p}(\C)$ consisting diagonal matrices and
$N_{\mathbb{C}}^{+}$ (resp.~$N_{\mathbb{C}}^{-}$) is the group of upper (resp.~lower) unitriangular matrices.
Then,
we have
$B_{\mathbb{C}}/N_{\mathbb{C}}\simeq T_{\mathbb{C}}$.
Thus,
it is sufficient to prove 
$
H^{1}(\mathbb{R},T_{\mathbb{C}})=1
$
and 
$
H^{1}(\mathbb{R},N_{\mathbb{C}})=1
$
by Fact \ref{Fact-Galois-exact-seq-long}.
For this purpose,
by using Lemma 
\ref{Lem-formulae-of-tau-and-H},
we calculate
\begin{gather*}
T_{\mathbb{C}}^{-\tau}
=
	\{
	T(t,\overline{t})
	\mid
	t\in 
	GL_{1}(\mathbb{C})^{p}
	\},
	\quad
N_{\mathbb{C}}^{-\tau}	
=
	\{
	N(n_{+},n_{+}^{*})
	\mid
	n_{+}\in 
	N_{\mathbb{C}}^{+}
	\}
\qquad \text{ and }
\\
\tau(T(t,s))^{-1}
T(t,s)
=
T(
\overline{s}t,
\overline{t}s),
\\
\tau(N(n_{+},n_{-}))^{-1}
N(n_{+},n_{-})
=
N(n_{-}^{*}n_{+},
n_{+}^{*}n_{-}).
\end{gather*}
These two formulas imply that
for any 
$
	T(t,\overline{t})
	\in
	T_{\mathbb{C}}^{-\tau} 
$ and 
$
	N(n_{+},n_{+}^{*})
	\in 
	N_{\mathbb{C}}^{-\tau}
$,
we have
(see Definition \ref{Def-sim_gamma} for the notation $[\cdot]_{\tau}$)
\begin{equation*}
\begin{array}{rcrcl}
	\tau(T(t,I_{p}))^{-1}
	T(t,I_{p})
	&=&
	T(t,\overline{t})
	,
	&
	\text{namely,}
	&
	[T(t,\overline{t})]_{\tau}
	=
	[e]_{\tau},
	\\
	\tau(N(n_{+},I_{p}))^{-1}
	N(n_{+},I_{p})
	&=&
	N(n_{+},n_{+}^{*})
	,
	&
	\text{namely,}
	&
	[N(n_{+},n_{+}^{*})]_{\tau}
	=
	[e]_{\tau},
\end{array}
\end{equation*}
where $e$ denotes the identity element.
These formulas  
mean that 
any elements of $T_{\mathbb{C}}^{-\tau}$
and $N_{\mathbb{C}}^{-\tau}$
are $\tau$-conjugate to the identity element,
which implies
$
H^{1}(\mathbb{R},T_{\mathbb{C}})=1
$
and
$
H^{1}(\mathbb{R},N_{\mathbb{C}})=1
$
by Lemma
\ref{Lem-Galois-1st-isom-Galois-conj}.

Next, we prove $ H^{1}(\mathbb{R},P_{\mathbb{C}}) = 1 $, 
the second equation of \eqref{eq-Lem-H^1(B_C,R)-H^1(P_C,R)}.
Define a subgroup of $P_{\mathbb{C}}$ by
\begin{align*}
	U_{\mathbb{C}}
	&\coloneqq
	\left\{
	U(u)
	\coloneqq
	\begin{pmatrix}
		I_{p} & u
		\\
		0 & I_{p}
	\end{pmatrix}
	\middle|
	u\in 
	\operatorname{M}_{p,p}(\mathbb{C})
	\right\}.
\end{align*}
Then,
we have
$P_{\mathbb{C}}/U_{\mathbb{C}}\simeq H_{\mathbb{C}}$.
Thus,
it is sufficient to prove 
$
H^{1}(\mathbb{R},H_{\mathbb{C}})=1
$ and $
H^{1}(\mathbb{R},U_{\mathbb{C}})=1
$ 
by Fact \ref{Fact-Galois-exact-seq-long}.
For this purpose,
we calculate
\begin{equation*}
	\tau(U(u))
	=
		\sqrt{-1}
	\begin{pmatrix}
		0 & -I_{p}
		\\ 
		I_{p} & 0
	\end{pmatrix}
	\begin{pmatrix}
		I_{p} & 0
		\\
		-u^{*} & I_{p}
	\end{pmatrix}
	\sqrt{-1}
	\begin{pmatrix}
		0 & -I_{p} 
		\\ 
		I_{p} & 0
	\end{pmatrix}
	=
	\begin{pmatrix}
		I_{p} & u^{*}
		\\
		0 & I_{p}
	\end{pmatrix}
    =U(u^{*}).
\end{equation*}
Therefore,
combining Lemma \ref{Lem-formulae-of-tau-and-H},
we get
\begin{gather*}
H_{\mathbb{C}}^{-\tau}
=
	\{
	\diag(g,g^{*})
	\mid
	g\in 
	GL_{p}(\mathbb{C})
	\},
	\quad
U_{\mathbb{C}}^{-\tau}	
=
	\{
	U(v)
	\mid
	v=-v^{*}
	\} 
\qquad
\text{ and }
\\
	\tau(\diag(g_{1},g_{2}))^{-1}
	\diag(g_{1},g_{2})
	=
	\diag(g_{2}^{*}g_{1},g_{1}^{*}g_{2})
	,
	\quad
\tau(U(u))^{-1}
U(u)
=
U(-u^{*}+u).
\end{gather*}
These two formulas imply that
for any 
$
	\diag(g,g^{*})
	\in
	H_{\mathbb{C}}^{-\tau}
	\text{ and }
	U(v)
	\in 
	U_{\mathbb{C}}^{-\tau}
$,
we have
\begin{equation*}
	\tau(\diag(g,I_{p}))^{-1}
	\diag(g,I_{p})
	=
	\diag(g,g^{*})
	,
	\quad
	\tau(U(2^{-1}v))^{-1}
	U(2^{-1}v)
	=
	U(v),
\end{equation*}
which
shows
$
H^{1}(\mathbb{R},H_{\mathbb{C}})=1
$ and $
H^{1}(\mathbb{R},U_{\mathbb{C}})=1
$
by Lemma
\ref{Lem-Galois-1st-isom-Galois-conj}.
\end{proof}

\begin{Corollary}\label{Cor-(G/P)^{tau}=G^{tau}/P^{tau}}
For $G_{\mathbb{C}}=GL_{2p}(\mathbb{C})$ and 
$P_{\mathbb{C}}$ 
in \eqref{eq-def-P_C-B_C}, 
we have
$
		(G_{\mathbb{C}}/P_{\mathbb{C}})^{\tau}
		=
		G_{\mathbb{C}}^{\tau}/P_{\mathbb{C}}^{\tau}
  = G/P
$.
\end{Corollary}

\begin{proof}
	This follows from 
	the exact sequence \eqref{eq-Fact-long-ex-seq-Galois},
	Lemma \ref{Lem-H^1(B_C,R)=1},
    and the equation \eqref{eq-def-G-H-P-B}.
\end{proof}

\section{Orbit decomposition over $\mathbb{C}$}
\label{Section-Orbit-decomp-over-C}
In this section,
we quote some results 
about 
$B_{\mathbb{C}}\backslash G_{\mathbb{C}}/P_{\mathbb{C}}$
from the previous work
\cite{FN-23-AIII},
where 
$G_{\mathbb{C}}, P_{\mathbb{C}}$
and
$B_{\mathbb{C}}$
are defined in Section \ref{Section-Over-Complex}.  
For 
$ \omega \in \operatorname{M}_{2p,p}^{\circ}
$, 
we denote the corresponding point $ \omega GL_{p}(\mathbb{C}) \in \operatorname{M}_{2p,p}^{\circ}/GL_{p}(\mathbb{C}) \simeq G_{\mathbb{C}}/P_{\mathbb{C}} $ by 
$ [\omega] $ (since $ G_{\mathbb{C}}/P_{\mathbb{C}} $ is the Grassmannian of $ p $-dimensional subspaces in $ \C^{2p} $, 
$ [\omega] \in G_{\mathbb{C}}/P_{\mathbb{C}} $ represents a subspace generated by its column vectors, or $ [ \omega] = \Image \omega $).

\subsection{Results for general $ p $}

We define two more involutions on
$G_{\mathbb{C}}=GL_{2p}(\mathbb{C})$
by
\begin{equation*}
	\varepsilon(g)\coloneqq 
	\overline{g},
	\quad
	\sigma(g)
	\coloneqq
	\Sigma_{p}^{-1}(g^{t})^{-1}\Sigma_{p},
\end{equation*}
where 
$\Sigma_{p}$ is given in 
\eqref{eq-def-Sigma-p-J-p}
and
$\overline{(\cdot)}$
is the complex conjugation for each entries.
Then,
we have
$
	\tau = \varepsilon \circ \sigma
$, 
%
and these three involutions $ \tau, \varepsilon $ and $ \sigma $ all commute with each other.

\begin{Fact}
\label{Fact-eqivalent-tau-fix-and-sigma-fix}
Let 
$G_{\mathbb{C}}, P_{\mathbb{C}}$
and
$B_{\mathbb{C}}$
be the groups defined in Section \ref{Section-Over-Complex},
and 
$\mathcal{O}$
a
$B_{\mathbb{C}}$-orbit 
on $G_{\mathbb{C}}/P_{\mathbb{C}}$.
Then,
we have
$
\mathcal{O}^{\sigma}\neq\emptyset
$
if and only if
$
\mathcal{O}^{\tau}\neq\emptyset.
$\end{Fact}

\begin{proof}
We choose a representative $ \omega $ of the orbit $ \mathcal{O} $ as given in \cite[Thm.~2.2 (1)]{FN-23-AIII}.  
Note that 
$
\mathcal{O}^{\sigma}\neq\emptyset
$
if and only if 
$ \sigma([\omega]) = [\omega] \in G_{\mathbb{C}}/P_{\mathbb{C}} $,  
which follows from the arguments in the proof of \cite[Prop.~4.5]{FN-21-embedding}.  
In the same way, we see that 
$
\mathcal{O}^{\tau}\neq\emptyset
$
if and only if
$ \tau([\omega]) = [\omega] $ (again by the same arguments in the proof of \cite[Prop.~4.5]{FN-21-embedding}).  
On the other hand,
all representatives $ \{ \omega \} $
given 
in 
\cite[Thm.~2.2 (1)]{FN-23-AIII}
have integer entries,
thus
they are fixed by $\varepsilon$.
Therefore $ \sigma([\omega]) = [\omega] $ if and only if $ \tau([\omega]) = [\omega] $ 
by 
the equality
$
	\tau = \varepsilon \circ \sigma
$, which proves the lemma.
\end{proof}

\subsection{The case of $p=2$}
In this section,
we specialize to the case $p=2$.

\begin{Fact}\label{Fact-list-of-sigma-fixed-p=2}
Let 
$G_{\mathbb{C}}, P_{\mathbb{C}}$
and
$B_{\mathbb{C}}$
be the groups defined in Section \ref{Section-Over-Complex},
and suppose that $p=2$.
Then,
all
representatives 
of
$B_{\mathbb{C}}$-orbits
on
$	\operatorname{M}_{2p,p}^{\circ}/GL_{p}(\mathbb{C})
	\simeq G_{\mathbb{C}}/P_{\mathbb{C}}$
given in
	\cite[Thm.~2.2 (1)]{FN-23-AIII}
with a $\tau$-fixed point
are listed in
the second entry from the left of
Table  \ref{Fig-list},
where we use the identification in 
Lemma \ref{Lem-G_C/P_C=M/GL}
and abbreviate
$[\omega] \in \operatorname{M}_{2p,p}^{\circ}/GL_{p}(\mathbb{C})$ to $\omega \in \operatorname{M}_{2p,p}^{\circ}$.
\end{Fact}

\begin{proof}
	The necessary and sufficient condition 
	for 
	a
	$B_{\mathbb{C}}$-orbit 
	on $G_{\mathbb{C}}/P_{\mathbb{C}}$
	to
	have a $\sigma$-fixed point
	is given in \cite[Thm.~4.6]{FN-21-embedding}.
	Thus,
	it is sufficient to apply this criterion 
	to 
	\cite[Ex.~2.4 (a)]{FN-23-AIII}
	by considering the equivalence 
	of Fact \ref{Fact-eqivalent-tau-fix-and-sigma-fix}
    (see also Remark \ref{Remark-Figure-FN} below).
\end{proof}

\begin{Remark}
\label{Remark-Figure-FN}
	In \cite{FN-23-AIII},
	the Borel subgroup of $H_{\mathbb{C}} \simeq GL_{p}(\mathbb{C})\times GL_{p}(\mathbb{C})$ is defined by
	$B_{p} \times B_{p}$
	while
	$B_{\mathbb{C}}$
	in this article 
	is defined by 
	$B_{p}\times \overline{B}_{p}$.  
	Thus,
	in order to apply
	the classification given in \cite[Ex.~2.4 (a)]{FN-23-AIII},
	it must be twisted by the longest element of the Weyl group of $1\times GL_{p}(\mathbb{C}) \subset H_{\mathbb{C}}$.
	More precisely,
	if 
	$\omega=
	M(
	\tau_{1},
	\tau_{2}	
	)
	$ 
	is the representative 
	(see
	\eqref{eq-def-LM-and-M(C,D)} for the notation) 
	of some
	$(B_{p}\times B_{p})$-orbit 
	corresponding to
	a 
	parameter  
	of 
	\cite[Fig.~2]{FN-23-AIII},
    then,
	$\omega=
	M(
	\tau_{1},
	w_{0}\tau_{2}w_{0}^{-1}	
	)
	$
	corresponds to 
	our $(B_{p}\times \overline{B}_{p})$-orbit,
	where 
	$\omega\in \mathfrak{T}_{(p,p),p}$
	(see \cite[Sect.~2.1]{FN-23-AIII} for the notation $\mathfrak{T}_{(p,p),p}$ and we identify $\omega$ with $[\omega] \in \operatorname{M}_{2p,p}^{\circ}/GL_{p}(\mathbb{C})
	\simeq  G_{\mathbb{C}}/P_{\mathbb{C}}$ by Lemma \ref{Lem-G_C/P_C=M/GL})
	and
	$w_{0}$ is the longest element of the Weyl group of $GL_{p}(\mathbb{C})$ 
which interchanges $ i $ and $ p - i + 1 \; (1\leq i \leq p)$.
\skipover{
	\begin{equation*}
		w_{0}
		\coloneqq
		\begin{pmatrix}
			& & 1 \\
			& \rotatebox{45}{$\dots$} &
			\\
			1 & &
		\end{pmatrix}.
	\end{equation*}
}	
	Moreover,
	in 
	\cite[Fig.~2]{FN-23-AIII},
	``$\dim$'' represents the dimension of  $H_{\mathbb{C}}$-orbits in 
	the double flag variety
	$H_{\mathbb{C}}/B_{\mathbb{C}}\times G_{\mathbb{C}}/P_{\mathbb{C}}$
	while,
	in Table  \ref{Fig-list},
	``$\dim$'' represents that of
	$B_{\mathbb{C}}$-orbits in
	$G_{\mathbb{C}}/P_{\mathbb{C}}$.
	
	Finally,
	there exists one error
	in 
	\cite[Fig.~2]{FN-23-AIII}.
	The central figure of \cite[Fig.~2]{FN-23-AIII}, 
in the row of
the $4$-dimensional orbits, should be replaced with
$
\vcenter{
		\xymatrix@C=0pt@R=0pt{
		\bullet 
		&
		\text{\textcircled{$\bullet$}}
		\phantom{.} 
		\\
		\bullet
		&
		\text{\textcircled{$\bullet$}}.
		}
}
$
\end{Remark}

\begin{Definition}\label{Def-Omega}
We write $\Omega \subset \operatorname{M}_{2p,p}^{\circ}$
for a set of all $\omega \in \operatorname{M}_{2p,p}^{\circ}$
listed in the second entry from the left of Table  \ref{Fig-list}.  
Hence $ \Omega $ is the set of complete representatives 
of $B_{\mathbb{C}}$-orbits in  
$\operatorname{M}_{2p,p}^{\circ}/GL_{p}(\mathbb{C}) \simeq G_{\mathbb{C}}/P_{\mathbb{C}}$
having a $\tau$-fixed point
by Fact \ref{Fact-list-of-sigma-fixed-p=2}.  
\end{Definition}

Recall that, for $ \omega \in \operatorname{M}_{2p,p}^{\circ} $, 
we denote the corresponding point $ \omega GL_{p}(\mathbb{C}) \in \operatorname{M}_{2p,p}^{\circ}/GL_{p}(\mathbb{C}) $ by $ [\omega] $. 

\begin{Corollary}
\label{Cor-B-G/P=coprod-H^1}
Let  
$G, P$ and $B$
be the groups defined in Section \ref{Section-Over-Real}
and assume $p=2$.
Then, we have
\begin{equation}
	B \backslash G/P
	\simeq 
	\coprod
	_{\omega \in \Omega }
	H^{1}(\mathbb{R},(B_{\mathbb{C}})_{[\omega]}),
\label{eq-Cor-p=2-B-G-P=H^{1}}
\end{equation}
where
$(B_{\mathbb{C}})_{[\omega]}$
is the stabilizer of
$[\omega] \in \operatorname{M}_{2p,p}^{\circ}/GL_{p}(\mathbb{C})$ 
in $B_{\mathbb{C}}$.
\end{Corollary}

\begin{proof}
Since $
H^{1}(\mathbb{R},B_{\mathbb{C}})=1
$
by Lemma \ref{Lem-H^1(B_C,R)=1}, 
we use Corollary
\ref{Cor-orbit-decomp-by-Galois-cohomology}
to obtain 
$
	B_{\mathbb{C}}^{\tau}
	\backslash
	(G_{\mathbb{C}}/P_{\mathbb{C}})^{\tau}
	\simeq 
 \linebreak
	\coprod_{\omega\in\Omega}
	H^{1}(\mathbb{R},(B_{\mathbb{C}})_{[\omega]})
$.  
This completes the proof by Corollary \ref{Cor-(G/P)^{tau}=G^{tau}/P^{tau}}.
\end{proof}

\section{Proof of Theorem \ref{Them-Main-in-intro}}
\label{Section-proof-of-main-thm}

In this section,
we give a proof of Theorem \ref{Them-Main-in-intro} 
by
computing the right-hand side 
of 
\eqref{eq-Cor-p=2-B-G-P=H^{1}}.
It is sufficient to prove that 
$ H^{1}(\mathbb{R},(B_{\mathbb{C}})_{[\omega]}) $ 
corresponds bijectively to the representatives of
$B$-orbits in $G/P\simeq \operatorname{LM}_{2p,p}^{\circ}/GL_{p}(\mathbb{C})$
given in Table  \ref{Fig-list}.

We only prove it
in the case where 
$
	\omega = \omega_{0}
	\coloneqq
	M\left(
	\begin{pmatrix}
		1 & 0 \\
		0 & 0 \\
	\end{pmatrix}
	,
	\begin{pmatrix}
		1 & 0 \\
		0 & 1
	\end{pmatrix}
	\right)
$ 
(we use the notation \eqref{eq-def-LM-and-M(C,D)}),
which is the fourth element from the bottom
of Table  \ref{Fig-list}.
The other cases 
can be proved similarly.  

First,
we calculate
\begin{align*}
	&
	\diag\left(
	\begin{pmatrix}
		t_{1} & n_{1}  \\
		0     & s_{1}  \\
	\end{pmatrix}
	,
	\begin{pmatrix}
		t_{2} & 0 \\
		n_{2} & s_{2}
	\end{pmatrix}
	\right)
	\omega_{0}
	\\
	=&
	M
	\left(
	\begin{pmatrix}
		t_{1}t_{2}^{-1} & 0 \\
		0 & 0 \\
	\end{pmatrix}
	,
	\begin{pmatrix}
		1 & 0 \\
		0 & 1
	\end{pmatrix}
	\right)
	\begin{pmatrix}
		t_{2} & 0 \\
		n_{2} & s_{2}
	\end{pmatrix}
	\in 
	M
	\left(
	\begin{pmatrix}
		t_{1}t_{2}^{-1} & 0 \\
		0 & 0 \\
	\end{pmatrix}
	,
	\begin{pmatrix}
		1 & 0 \\
		0 & 1
	\end{pmatrix}
	\right)
	GL_{p}(\mathbb{C}).
\end{align*}
This implies that
the stabilizer 
$
B_{\mathbb{C}}'
\coloneqq
(B_{\mathbb{C}})_{[\omega_0]}$
at $ [\omega_0] = \omega_{0} GL_{p}(\mathbb{C})
\in \operatorname{M}_{2p,p}^{\circ}/GL_{p}(\mathbb{C})$
of $B_{\mathbb{C}}$
is equal to what is given in Table  \ref{Fig-list}.

Next, in order to determine the Galois cohomology of $
(B_{\mathbb{C}})_{[\omega_0]}
$,
we calculate 
by using Lemma \ref{Lem-formulae-of-tau-and-H},
\begin{equation*}
    B_{\mathbb{C}}'^{-\tau}
    =
    \left\{
    \diag
    \left(
    \begin{pmatrix}
        x & n \\
        0 & s 
    \end{pmatrix}
    ,
    \begin{pmatrix}
        x & 0 \\
        \overline{n} & \overline{s}
    \end{pmatrix}
    \right)
    \middle|
    \begin{array}{c}
         x \in GL_{1}(\mathbb{R})
         \\
         s \in GL_{1}(\mathbb{C}), 
         \:
         n \in \mathbb{C}
    \end{array}
    \right\}.
\end{equation*}
The similar arguments in Lemma \ref{Lem-H^1(B_C,R)=1} tell that 
\begin{equation}
H^{1}(\mathbb{R},B_{\mathbb{C}}')
        \simeq
        B_{\mathbb{C}}'^{-\tau}/\underset{\tau}{\sim} \;\;\xrightarrow{\;\sim\;\;} \{\pm1\}, 
\qquad 
    \diag
    \left(
    \begin{pmatrix}
        x & n \\
        0 & s 
    \end{pmatrix}
    ,
    \begin{pmatrix}
        x & 0 \\
        \overline{n} & \overline{s}
    \end{pmatrix}
    \right) \mapsto x/|x|
\skipover{
	\begin{array}{ccc}
	H^{1}(\mathbb{R},T_{\mathbb{C}}')
	\simeq 
	T_{\mathbb{C}}'^{-\tau}/\underset{\tau}{\sim}
	&
	\xrightarrow{\;\sim\;\;}
	&
	\{\pm1\}
	\\
	\rotatebox{90}{$\in$}
	&&
	\rotatebox{90}{$\in$}
	\\
	T(x,s,x,\overline{s})
	&
	\mapsto 
	&
	\frac{|x|}{x},
	\end{array}
}
\label{eq-isom-H^{1}(T,R)=+-1}
\end{equation}
under
the identification
\eqref{eq-Lem-H^1-isom-conj}. Thus, we get
$
	H^{1}(\mathbb{R},B_{\mathbb{C}}')
	=\{\pm 1\}, 
$
which coincides with what is given in Table  \ref{Fig-list}.

Finally, we determine 
the 
orbits on 
$G/P\simeq \operatorname{LM}_{2p,p}^{\circ}/GL_{p}(\mathbb{C})$
corresponding
to each elements of 
$H^{1}(\mathbb{R},B_{\mathbb{C}}')
	\simeq \{\pm1\}$.
For $t_{1},t_{2},s_{1},s_{2} \in GL_{1}(\C)$,
we write
\begin{equation*}
    T(t_{1},t_{2},s_{1},s_{2})
    \coloneqq
    T(\diag(t_{1},t_{2}), \diag(s_{1},s_{2}) ),
\end{equation*}
where we use the notation defined in \eqref{eq-def-T-and-N}.
Then, by the definition of the map $ \delta $ in \eqref{eq-Fact-hom-sp=1st-Galois}
(and \eqref{eq-Lem-H(b,c)^{-1}}),
we have
\begin{equation}
\begin{aligned}
	\delta(
	B_{\mathbb{C}}^{\tau}T(-1,1,1,1)P_{\mathbb{C}})
	&=
	[T(-1,1,1,1)^{-1}
	\tau(T(-1,1,1,1))]
	_{\tau}
	\\
	&=
	[T(-1,1,1,1)
	T(1,1,-1,1)]_{\tau}
	=
	[T(-1,1,-1,1)]_{\tau},
\end{aligned}
\label{eq-Sect-CC-delta(T)=...}
\end{equation}
where $[\cdot]_{\tau}$ is defined in 
Definition \ref{Def-sim_gamma}.
The right-most term of
\eqref{eq-Sect-CC-delta(T)=...} 
corresponds
to
$-1$
under the isomorphism
\eqref{eq-isom-H^{1}(T,R)=+-1}.
Thus,
Equation \eqref{eq-Fact-hom-sp=1st-Galois} tells that 
the $B_{\mathbb{C}}^{\tau}$-orbit 
on
$
(G_{\mathbb{C}}/P_{\mathbb{C}})^{\tau}
=
G/P\simeq \operatorname{LM}_{2p,p}^{\circ}/GL_{p}(\mathbb{C})$
corresponding to
$-1\in H^{1}(\mathbb{R},B_{\mathbb{C}}')$
is generated by 
\begin{equation*}
	T(-1,1,1,1)\omega_{0} GL_{p}(\mathbb{C})
	=
	M
	\left(
	\begin{pmatrix}
		-1 & 0 \\
		0 & 0 
	\end{pmatrix}
	,
	\begin{pmatrix}
		1 & 0 \\
		0 & 1
	\end{pmatrix}
	\right)
	GL_{p}(\mathbb{C}).
\end{equation*}
Moreover,
the 
orbit on $G/P\simeq \operatorname{LM}_{2p,p}^{\circ}/GL_{p}(\mathbb{C})$
corresponding
to 
$1\in H^{1}(\mathbb{R},B_{\mathbb{C}}') $
is $ B \omega_{0} P $.
These coincides with the corresponding entries in Table  \ref{Fig-list}, which 
completes the proof.

\begin{table}[hptb]
\resizebox{.75\textwidth}{!}{
\(
\begin{array}{|c|c|c|c|c|}
\hline
\dim 
& 
\omega
& 
(B_{\mathbb{C}})_{[\omega]}
&
H^{1}(\mathbb{R},
(B_{\mathbb{C}})_{[\omega]}
)
&
\begin{array}{c}
\text{representatives of}
\\
B\backslash (B_{\mathbb{C}}/(B_{\mathbb{C}})_{[\omega]})^{\tau} 
\end{array}
\\	
\hline
4
& 
\begin{pmatrix}
	1 & 0 \\
	0 & 1 \\
	1 & 0 \\
	0 & 1 
\end{pmatrix}
&
\left\{
\begin{pmatrix}
	t & 0 & 0 & 0 \\
	0 & s & 0 & 0 \\
	0 & 0 & t & 0 \\
	0 & 0 & 0 & s
\end{pmatrix}
\right\}
&
\{\pm 1\}^{2}
&
\begin{pmatrix}
	\pm1 & 0 \\
	0 & \pm1 \\
	1 & 0 \\
	0 & 1 
\end{pmatrix}
\\
\hline
3
& 
\begin{pmatrix}
	1 & 0 \\
	0 & 1 \\
	1 & 0 \\
	0 & 0 
\end{pmatrix}
&
\left\{
\begin{pmatrix}
	t & 0 & 0 & 0 \\
	0 & s_{1} & 0 & 0 \\
	0 & 0 & t & 0 \\
	0 & 0 & 0 & s_{2} 
\end{pmatrix}
\right\}
&
\{\pm 1\}
&
\begin{pmatrix}
	1 & 0 \\
	0 & 1 \\
	\pm1 & 0 \\
	0 & 0 
\end{pmatrix}
\\
\hline
3
& 
\begin{pmatrix}
	0 & 1 \\
	1 & 0 \\
	1 & 0 \\
	0 & 1 
\end{pmatrix}
&
\left\{
\begin{pmatrix}
	t & n & 0 & 0 \\
	0 & s & 0 & 0 \\
	0 & 0 & s & 0 \\
	0 & 0 & n & t
\end{pmatrix}
\right\}
&
1
& 
\begin{pmatrix}
	0 & 1 \\
	1 & 0 \\
	1 & 0 \\
	0 & 1 
\end{pmatrix}
\\
\hline
3
& 
\begin{pmatrix}
	0 & 0 \\
	0 & 1 \\
	1 & 0 \\
	0 & 1 
\end{pmatrix}
&
\left\{
\begin{pmatrix}
	t_{1} & 0 & 0 & 0 \\
	0 & s & 0 & 0 \\
	0 & 0 & t_{2} & 0 \\
	0 & 0 & 0 & s 
\end{pmatrix}
\right\}
&
\{\pm 1\}
&
\begin{pmatrix}
	0 & 0 \\
	0 & \pm1 \\
	1 & 0 \\
	0 & 1 
\end{pmatrix}
\\
\hline
2
& 
\begin{pmatrix}
	0 & 0 \\
	0 & 1 \\
	1 & 0 \\
	0 & 0 
\end{pmatrix}
&
\left\{
\begin{pmatrix}
	t_{1} & 0 & 0 & 0 \\
	0 & s_{1} & 0 & 0 \\
	0 & 0 & t_{2} & 0 \\
	0 & 0 & 0 & s_{2} 
\end{pmatrix}
\right\}
&
1
&
\begin{pmatrix}
	0 & 0 \\
	0 & 1 \\
	1 & 0 \\
	0 & 0 
\end{pmatrix}
\\
\hline
1
& 
\begin{pmatrix}
	1 & 0 \\
	0 & 1 \\
	0 & 0 \\
	0 & 1 
\end{pmatrix}
&
\left\{
\begin{pmatrix}
	t_{1} & n_{1} & 0 & 0 \\
	0 & s & 0 & 0 \\
	0 & 0 & t_{2} & 0 \\
	0 & 0 & n_{2} & s 
\end{pmatrix}
\right\}
&
\{\pm 1\}
&
\begin{pmatrix}
	1 & 0 \\
	0 & 1 \\
	0 & 0 \\
	0 & \pm1 
\end{pmatrix}
\\
\hline
1
& 
\begin{pmatrix}
	1 & 0 \\
	0 & 0 \\
	1 & 0 \\
	0 & 1 
\end{pmatrix}
&
\left\{
\begin{pmatrix}
	t & n_{1} & 0 & 0 \\
	0 & s_{1} & 0 & 0 \\
	0 & 0 & t & 0 \\
	0 & 0 & n_{2} & s_{2} 
\end{pmatrix}
\right\}
&
\{\pm 1\}
&
\begin{pmatrix}
	\pm1 & 0 \\
	0 & 0 \\
	1 & 0 \\
	0 & 1 
\end{pmatrix}
\\
\hline
0
& 
\begin{pmatrix}
	1 & 0 \\
	0 & 1 \\
	0 & 0 \\
	0 & 0 
\end{pmatrix}
&
B_{\mathbb{C}}
&
1
& 
\begin{pmatrix}
	1 & 0 \\
	0 & 1 \\
	0 & 0 \\
	0 & 0 
\end{pmatrix}
\\
\hline
0
& 
\begin{pmatrix}
	1 & 0 \\
	0 & 0 \\
	0 & 0 \\
	0 & 1 
\end{pmatrix}
&
B_{\mathbb{C}}
&
1
& 
\begin{pmatrix}
	1 & 0 \\
	0 & 0 \\
	0 & 0 \\
	0 & 1 
\end{pmatrix}
\\
\hline
0
& 
\begin{pmatrix}
	0 & 0 \\
	0 & 0 \\
	1 & 0 \\
	0 & 1 
\end{pmatrix}
&
B_{\mathbb{C}}
&
1
& 
\begin{pmatrix}
	0 & 0 \\
	0 & 0 \\
	1 & 0 \\
	0 & 1 
\end{pmatrix}
\\
\hline
\end{array}	
\)
}

\bigskip
{\small
\caption{In this list,
we always assume that
$t,t_{1},t_{2},s,s_{1},s_{2} \in GL_{1}(\mathbb{C})$
and
$n,n_{1},n_{2}\in \mathbb{C}$.
Moreover,
we identify
$
G_{\mathbb{C}}/P_{\mathbb{C}} \simeq \operatorname{M}_{2p,p}^{\circ}/GL_{p}(\mathbb{C})$
(see Lemma \ref{Lem-G_C/P_C=M/GL})
and abbreviate
$[\omega] \in \operatorname{M}_{2p,p}^{\circ}/GL_{p}(\mathbb{C})$ to $\omega \in \operatorname{M}_{2p,p}^{\circ}$.\label{Fig-list}}
}
\end{table}

\appendix
\section{Galois cohomology}\label{Section-Galois-cohomolgy}

In this section,
we 
summarize
some of
basic results about 
the theory of the Galois cohomology,
which are used in this article.
Main references are
\cite{BS-GC,Serre-GC}.
\begin{Convention}
\label{Conv-defined-over-R}
	In this article,
	\textit{an algebraic group $(G_{\mathbb{C}},\gamma)$ defined over $\mathbb{R}$}
	means that
	$G_{\mathbb{C}}$ is a complex algebraic group with an action of the Galois group 
	$\operatorname{Gal}(\mathbb{C}/\mathbb{R})$ of the field extension $\mathbb{C}/\mathbb{R}$, and
	$\gamma\colon G_{\mathbb{C}}\to G_{\mathbb{C}}$ denotes the action of the nontrivial element of
	$\operatorname{Gal}(\mathbb{C}/\mathbb{R})$.
    Moreover, we say 
    \textit{$(H_{\mathbb{C}},\gamma)$ is a subgroup of $(G_{\mathbb{C}},\gamma)$ defined over $\mathbb{R}$}, 
    when $H_{\mathbb{C}}$ is an algebraic subgroup of $G_{\mathbb{C}}$ that is also $\gamma$-stable.
    Thus, 
    $(H_{\mathbb{C}},\gamma)$ is also an algebraic group defined over $\mathbb{R}$.
\end{Convention}

\begin{Definition}[{\cite[Sect.~1.2]{BS-GC},
\cite[Chap.~I, \S5, 5.1]{Serre-GC}}]
\label{Def-Galois-coh}
Let $(G_{\mathbb{C}},\gamma)$ be an algebraic group 
defined over $\mathbb{R}$.
We set
$
	Z^{1}(\mathbb{R},G_{\mathbb{C}})
	\coloneqq
	\{g\colon 
	\operatorname{Gal}(\mathbb{C}/\mathbb{R})
	\to 
	G_{\mathbb{C}}\mid 
	g\text{ satisfies }
	(*)
	\}
$,
where 
the condition $(*)$ is given by
\begin{equation*}
	(*)
	\qquad
	\forall x, y\in 
	\operatorname{Gal}(\mathbb{C}/\mathbb{R}), \quad
		g_{x y} = g_{x}\, x(g_{y}).
\end{equation*}
Moreover, we define an equivalence relation $\sim$ on $Z^{1}(\mathbb{R}, G_{\mathbb{C}})$ by
\begin{equation*}
	g \sim g'
	\overset{\text{def}}{\iff}
	\exists h\in G,\;
	\forall x\in 
	\operatorname{Gal}(\mathbb{C}/\mathbb{R}),
 \quad
	g_{x} = h^{-1} g'_{x}\, x(h)
\end{equation*}
and set
$
	H^{1}(\mathbb{R},G_{\mathbb{C}})
	\coloneqq
	Z^{1}(\mathbb{R},G_{\mathbb{C}})
	/
	\sim $.
\end{Definition}

\begin{Definition}
\label{Def-sim_gamma}
Let $(G_{\mathbb{C}},\gamma)$ be an algebraic group 
defined over $\mathbb{R}$.
Then,
we set
\begin{equation}
	G_{\mathbb{C}}^{-\gamma}
	\coloneqq 
	\{g\in G_{\mathbb{C}}\mid \gamma(g)=g^{-1}\}.
\label{eq-def-G^{-gamma}}
\end{equation}
Moreover,
we define an equivalence relation $\underset{\gamma}{\sim}$ on $G_{\mathbb{C}}^{-\gamma}$ by
$\gamma$-conjugation:
\begin{equation*}
	g \underset{\gamma}{\sim} g'
	\quad
	\overset{\text{def}}{\iff}
	\quad
	\exists h\in G_{\mathbb{C}}
	\text{ such that }
	g=h^{-1}g'\gamma(h).
\end{equation*}
We also write $[g]_{\gamma}$
for the equivalence class of $g\in G_{\mathbb{C}}^{-\gamma}$
under the relation $\underset{\gamma}{\sim}$.
\end{Definition}

\begin{Lemma}
\label{Lem-Galois-1st-isom-Galois-conj}
Let $(G_{\mathbb{C}},\gamma)$ be an algebraic group defined over $\mathbb{R}$.
Then,
we have an isomorphism:
\begin{equation}
	H^{1}(\mathbb{R},G_{\mathbb{C}})
	\xrightarrow{\;\sim\;\;}
	G_{\mathbb{C}}^{-\gamma}
	/
	\underset{\gamma}{\sim}.
\label{eq-Lem-H^1-isom-conj}
\end{equation}	
\end{Lemma}

\begin{proof}
For any
$g\in Z^{1}(\mathbb{R},G_{\mathbb{C}})$,
the condition $(*)$ in
Definition \ref{Def-Galois-coh}
forces that 
the value of $g$ at the identity element 
(resp.~nontrivial element)
of $\operatorname{Gal}(\mathbb{C}/\mathbb{R})$
is
the identity element of $G_{\mathbb{C}}$
(resp.~an element of $G_{\mathbb{C}}^{-\gamma}$).
Thus,
the evaluation map 
at the nontrivial element
of $\operatorname{Gal}(\mathbb{C}/\mathbb{R})$
induces the desired isomorphism.
\end{proof}

\begin{Fact}[{\cite[Prop.~1.12, Cor.~1.13]{BS-GC},
\cite[Prop.~36, Cor.~1 for Prop.~36]{Serre-GC}}]
\label{Fact-Galois-exact-seq}
Let $(G_{\mathbb{C}},\gamma)$ be an algebraic group defined over $\mathbb{R}$,
$(H_{\mathbb{C}},\gamma)$ its algebraic subgroup defined over $\mathbb{R}$.
Then, there exists an exact sequence 
of pointed sets:
\begin{equation}
	1
	\to 
	H_{\mathbb{C}}^{\gamma}
	\to 
	G_{\mathbb{C}}^{\gamma}
	\to 
	(G_{\mathbb{C}}/H_{\mathbb{C}})^{\gamma}
	\overset{\delta}{\to} 
	H^{1}(\mathbb{R},H_{\mathbb{C}})
	\to 
	H^{1}(\mathbb{R},G_{\mathbb{C}}),
\label{eq-Fact-long-ex-seq-Galois}
\end{equation}
where $\delta$ is the connecting homomorphism
and
$(G_{\mathbb{C}}/H_{\mathbb{C}})^{\gamma}$ is the subset of $\gamma$-invariant elements.
Moreover,
if $H^{1}(\mathbb{R},G_{\mathbb{C}})$ is trivial,
we have a bijection 
\begin{equation}
\delta : G_{\mathbb{C}}^{\gamma} \backslash (G_{\mathbb{C}}/H_{\mathbb{C}})^{\gamma} \xrightarrow{\;\sim\;\;} 
H^{1}(\mathbb{R},H_{\mathbb{C}}), \qquad
G_{\mathbb{C}}^{\gamma} gH_{\mathbb{C}} \mapsto [g^{-1}\gamma(g)]_{\gamma}
\skipover{
\begin{array}{ccc}
	G_{\mathbb{C}}^{\gamma}	\backslash (G_{\mathbb{C}}/H_{\mathbb{C}})^{\gamma}
	&
	\underset{\delta
	}{\xrightarrow{\;\sim\;\;}}
	&
	H^{1}(\mathbb{R},H_{\mathbb{C}})
	\\
	\rotatebox{90}{$\in$}
	&&
	\rotatebox{90}{$\in$}
	\\
	G_{\mathbb{C}}^{\gamma}	gH_{\mathbb{C}}
	&
	\mapsto
	&
	[g^{-1}\gamma(g)]_{\gamma}
\end{array}
}
\label{eq-Fact-hom-sp=1st-Galois}
\end{equation}
under the 
identification \eqref{eq-Lem-H^1-isom-conj},
where 
$g$ is an element of $G_{\mathbb{C}}$
satisfying $\gamma(gH_{\mathbb{C}})=gH_{\mathbb{C}}$.
\end{Fact}

\begin{Corollary}
\label{Cor-orbit-decomp-by-Galois-cohomology}
Let $(G_{\mathbb{C}},\gamma)$ be an algebraic group defined over $\mathbb{R}$,
$(P_{\mathbb{C}},\gamma)$
and
$(B_{\mathbb{C}},\gamma)$
algebraic subgroups defined over $\mathbb{R}$
of $(G_{\mathbb{C}},\gamma)$
satisfying
$
	H^{1}(\mathbb{R},B_{\mathbb{C}})=1
$.
Let us choose a complete set of representatives $\Xi'\subset G_{\mathbb{C}}$ 
of the double coset decomposition 
$ 
	B_{\mathbb{C}}\backslash G_{\mathbb{C}}/P_{\mathbb{C}} 
$.
\skipover{
Assume that there exists a subset $\Xi'\subset G_{\mathbb{C}}$ such that the orbit decomposition 
of $B_{\mathbb{C}}\backslash G_{\mathbb{C}}/P_{\mathbb{C}}$ is given by
\begin{equation*}
	B_{\mathbb{C}}\backslash G_{\mathbb{C}}/P_{\mathbb{C}}
	\simeq 
	\coprod\nolimits_{\xi'\in \Xi}
	B_{\mathbb{C}}\xi' P_{\mathbb{C}}.
\end{equation*}
}
Define a subset $\Xi\subset\Xi'$
by 
$
	\Xi
	\coloneqq
	\{\xi\in\Xi'\mid
	(B_{\mathbb{C}}\xi P_{\mathbb{C}})^{\gamma}\neq \emptyset\}
$.
Retaking each representative of
$B_{\mathbb{C}}\xi P_{\mathbb{C}}$ for $\xi \in \Xi$ if necessary,
we can assume that $\gamma(\xi P_{\mathbb{C}})=\xi P_{\mathbb{C}}$ for any $\xi \in \Xi$.
Then, we have a bijection
\begin{equation*}
	B_{\mathbb{C}}^{\gamma}
	\backslash
	(G_{\mathbb{C}}/P_{\mathbb{C}})^{\gamma}
	\simeq 
	\coprod\nolimits_{\xi\in\Xi}
	H^{1}(\mathbb{R},(B_{\mathbb{C}})_{\xi P_{\mathbb{C}}}),
\end{equation*}
where
$(B_{\mathbb{C}})_{\xi P_{\mathbb{C}}}$
is the stabilizer in $B_{\mathbb{C}}$
at $\xi P_{\mathbb{C}}\in G_{\mathbb{C}}/P_{\mathbb{C}}$.
\end{Corollary}

\begin{proof}
As a set,
$G_{\mathbb{C}}/P_{\mathbb{C}}$ is a disjoint union of double cosets:
$
	G_{\mathbb{C}}/P_{\mathbb{C}}
	=
	\coprod\nolimits_{\xi'\in \Xi'}
	B_{\mathbb{C}}\xi' P_{\mathbb{C}}
$.
Therefore,
taking $\gamma$-fixed points,
we have
\begin{equation*}
	(G_{\mathbb{C}}/P_{\mathbb{C}})^{\gamma}
	=
	\coprod\nolimits_{\xi'\in \Xi'}
	(B_{\mathbb{C}}\xi' P_{\mathbb{C}})^{\gamma}
	=
	\coprod\nolimits_{\xi\in \Xi}
	(B_{\mathbb{C}}\xi P_{\mathbb{C}})^{\gamma}
	.
\end{equation*}
Since 
$B_{\mathbb{C}}\xi P_{\mathbb{C}}/ P_{\C} \simeq B_{\mathbb{C}}/(B_{\mathbb{C}})_{\xi P_{\C}}$
as $B_{\mathbb{C}}$-homogeneous spaces,
the above formula becomes 
\begin{equation*}
	(G_{\mathbb{C}}/P_{\mathbb{C}})^{\gamma}
	=	
	\coprod\nolimits_{\xi\in \Xi}
	(B_{\mathbb{C}}/(B_{\mathbb{C}})_{\xi P_{\C}})^{\gamma},
\end{equation*}
which implies
\begin{equation*}
	B_{\mathbb{C}}^{\gamma}\backslash
	(G_{\mathbb{C}}/P_{\mathbb{C}})^{\gamma}
	=	
	\coprod\nolimits_{\xi \in \Xi}
	B_{\mathbb{C}}^{\gamma}\backslash
	(B_{\mathbb{C}}/(B_{\mathbb{C}})_{\xi P_{\C}})^{\gamma}
	\simeq 
	\coprod\nolimits_{\xi \in\Xi}
	H^{1}(\mathbb{R},(B_{\mathbb{C}})_{\xi P_{\C}})
\end{equation*}
by the isomorphism
\eqref{eq-Fact-hom-sp=1st-Galois}.
This completes the proof.
\end{proof}

\begin{Fact}[{\cite[Prop.~1.17]{BS-GC},
\cite[Prop.~38]{Serre-GC}}]
\label{Fact-Galois-exact-seq-long}
Let $(G_{\mathbb{C}},\gamma)$ be an algebraic group defined over $\mathbb{R}$,
$(H_{\mathbb{C}},\gamma)$ its algebraic subgroup defined over $\mathbb{R}$.  
If $G_{\mathbb{C}}/H_{\mathbb{C}}$ is a group (i.e., if $ H_{\mathbb{C}} $ is a normal subgroup),
we can add one more term to the above long exact sequence \eqref{eq-Fact-long-ex-seq-Galois} to get
\begin{equation*}
	1
	\to 
	H_{\mathbb{C}}^{\gamma}
	\to 
	G_{\mathbb{C}}^{\gamma}
	\to 
	(G_{\mathbb{C}}/H_{\mathbb{C}})^{\gamma}
	\overset{\delta}{\to} 
	H^{1}(\mathbb{R},H_{\mathbb{C}})
	\to 
	H^{1}(\mathbb{R},G_{\mathbb{C}})
	\to 
	H^{1}(\mathbb{R},G_{\mathbb{C}}/H_{\mathbb{C}})
	.
\end{equation*}
\end{Fact}

\end{document}